\documentclass[11pt,dvips]{article}
\usepackage{latexsym}
\usepackage{amsmath,amsthm,amsfonts,amssymb}
\usepackage{enumerate}
\usepackage{graphicx}
\usepackage[all]{xy}

\newtheorem{thm}{Theorem}[section]
 
\newtheorem{cor}[thm]{Corollary}
\newtheorem{prop}[thm]{Proposition} 
\newtheorem{lem}[thm]{Lemma} 
\newtheorem*{qst*}{Question} 

\theoremstyle{remark}
\newtheorem*{rem*}{Remark}
\newtheorem{rem}[thm]{Remark}
\newtheorem*{example*}{Example}

\renewcommand{\phi}{\varphi} 
\newcommand{\m} {^{-1}}

\newcommand {\cals} {{\mathcal {S}}}   

\newcommand {\N} {{\mathbb {N}}} 
\newcommand {\Z} {{\mathbb {Z}}}

\newcommand {\Q} {{\mathbb {Q}}}
\newcommand{\inc}{\subset}
\newcommand\Or{OR}
\newcommand\vor{VOR}

\usepackage{pdfsync}

\begin{document}
 
\title{Rank of mapping tori  and companion matrices}

\author{ Gilbert Levitt and Vassilis Metaftsis}
\date{}

\maketitle

\begin{abstract}
Given $\varphi\in GL(d,\Z)$, it is decidable whether the mapping torus $G=\Z^d\rtimes_\varphi\Z$ has rank 2 or not (i.e.\ whether $G$ may be generated by two elements); when it does, one may classify generating pairs up to Nielsen equivalence. If $\varphi$ has infinite order, the rank of $\Z^d\rtimes_{\varphi^n}\Z$ is at least 3 for all $n$ large enough; equivalently,  $\varphi^n$ is not conjugate to a companion matrix in $GL(d,\Z)$ if $n$ is large. 
\end{abstract}
\setcounter{tocdepth}{2}

\section{Introduction}

The rank of a finitely generated group is the minimum cardinality of a generating set. There are very few families of groups for which one knows how to compute the rank (see \cite{kw} and references therein),  and there exists no  algorithm computing the rank of a word-hyperbolic group \cite{sh}. 

By Grushko's theorem, rank is additive under free product. It does not behave as nicely under direct product, even when one of the factors is $\Z$: the solvable Baumslag-Solitar group $BS(1,2)=\langle a,t\mid tat\m=a^2\rangle$ and the product $BS(1,2)\times \Z$ both  have rank 2.

In this paper we consider semi-direct products $G=A\rtimes_\varphi\Z$ (also known as mapping tori), with the generator of the cyclic group $\Z$ acting on $A$ by some automorphism $\varphi\in Aut(A)$. This was motivated by the remark that, when $A$ is a free group $F_d$ and $\varphi$ has finite order in $Out(F_d)$, then $G$ is a generalized Baumslag-Solitar group and its rank may be computed \cite{le}. But we do not know how to compute the rank when $\varphi$ has infinite order. Abelianizing does not help much, so we ask:

\begin{qst*} Given $\varphi\in GL(d,\Z)$, can one compute the rank of $G=\Z^d\rtimes_\varphi\Z$?
\end{qst*}

We can prove:

\begin{thm} \label{decid} Given $\varphi\in GL(d,\Z)$,   one can decide whether $G=\Z^d\rtimes_\varphi\Z$ has rank 2 or not. 
\end{thm}

It turns out that the rank of $G$  is 1 plus the minimum number $k$ such that $\Z^d$ may be generated by $k$ orbits of $\varphi$ (i.e.\ there exist $g_1,\dots,g_k\in\Z^d$ such that the elements $\varphi^n(g_i)$, for $n\in\Z$ and $i=1,\dots, k$, generate $\Z^d$). 
In particular, $G$
has rank 2 if and only if $\Z^d$ may be generated by a single $\varphi$-orbit. This happens precisely when
$\varphi$ is conjugate to the companion matrix having the same characteristic polynomial. This may be decided since the conjugacy problem is solvable in $GL(d,\Z)$ \cite{grunewald}. 

Theorem \ref{decid}  extends to the case when $\varphi$ is an automorphism of an arbitrary finitely generated nilpotent group $A$.

When $G$ has rank 2, one can classify generating pairs up to Nielsen equivalence. In particular: 

\begin{thm}Suppose that $G=\Z^d\rtimes_\varphi\Z$ has rank 2.
 There are infinitely many Nielsen classes of generating pairs  if and only if the cyclic subgroup of $GL(d,\Z)$ generated by $\varphi$ has infinite index in its centralizer. 
\end{thm}

Our next result is motivated by the following theorem due to J.\ Souto:
\begin{thm}[\cite{souto}] Let $A$ be the fundamental group of a closed orientable surface of genus $g\ge 2$. Let $\varphi$ be an automorphism of $A$ representing a pseudo-Anosov mapping class. Then there exists $n_0$ such that  the rank of $G_n=A\rtimes_{\varphi^n}\Z$ is $2g+1$ for all $n\ge n_0$.
\end{thm}

We prove:
\begin{thm} \label{powe} Given $\varphi$ of infinite order in $ GL(d,\Z)$, with $d\ge2$,  there exists $n_0$ such that  the rank of $G_n=\Z^d\rtimes_{\varphi^n}\Z$ is $\ge3$ for all $n\ge n_0$.
\end{thm}

The   theorem   becomes false if the hypothesis that $\varphi$ has infinite order is dropped, or if 3 is replaced by 4. We do not know hypotheses that would guarantee that the rank is $d+1$ for $n$ large.

An equivalent formulation of Theorem \ref{powe} is:
\begin{thm} \label{pow2} Given a matrix $M$  of infinite order in $ GL(d,\Z)$, with $d\ge2$,  there exists $n_0$ such that  $M^n$ is not conjugate to a companion matrix  if $n\ge n_0$.
\end{thm}

Our proof is based on the Skolem-Mahler-Lech theorem on linear recurrent sequences \cite{rs}. There are alternative approaches based on equations in $S$-units and Baker's theory on linear forms in logarithms. They are due to Amoroso-Zannier \cite  {AZ} and yield uniformity: \emph{one may take $n_0=[C d^6(\log d)^6]$ where $C$ is a universal constant (independent of $M$).}

We conclude with a few   open questions. 

Our analysis on $\Z^d$ uses the Cayley-Hamilton theorem. This is not available in a non-abelian free group $F_d$. Given $\varphi\in Aut(F_d)$, can one decide whether $F_d$ may be generated by a single  $\varphi$-orbit? More basically: given $\varphi\in Aut(F_d)$ and $g\in F_d$, can one decide whether the $\varphi$-orbit of $g$ generates $F_d$?

What about ascending HNN extensions? For instance, let $\varphi$ be an injective endomorphism of $\Z^d$ (a matrix with integral entries and non-zero determinant). Let $G=\Z^d*_\varphi=\langle \Z^d,t\mid tgt\m=\varphi(g)\rangle$. Can one decide whether $G$ has rank 2?
\medskip

{\it \small
Acknowledgements. We wish to thank J.-L.\  Colliot-Th\'el\`ene, F.\  Grunewald, P.\ de la Harpe, G.\ Henniart,  and number theorists in Caen, in particular F.\  Amoroso, J.-P.\  Bezivin, D.\  Simon, for helpful conversations related to this work. The second author would also like to thank LMNO of Universit\'e de Caen for their hospitality during the
preparation of the present work.}

   \section{Generalities}
  
Let $A$ be a finitely generated group. The letters $a,b,v$ will always denote elements of $A$. We denote by $i_a$ the inner automorphism $v\mapsto ava\m$.

Given $\varphi
\in Aut(A)$, we let $G$  
be the mapping torus $G=A\rtimes_\varphi \Z=  \langle A,t\mid tat\m=\varphi(a)\rangle$. There is an exact sequence $1\to A\to G\to \Z\to 1$. Up to isomorphism, $G$ only depends on the image of $\varphi$ in $Out(A)$. Any $g\in G$ has   unique forms $at^n, t^na'$ with $n\in\Z$. 

If $N$ is a characteristic subgroup of $A$, we denote by $\bar \varphi$ the automorphism induced on $A/N$. There is an exact sequence $1\to N\to A\rtimes_\varphi\Z\to A/N\rtimes_{\bar\varphi}\Z\to 1$.

 The rank $rk(G)$ is the minimum cardinality of a generating set. We let $vrk(G)$ be the   minimum number of elements needed to generate a finite index subgroup: $vrk(G)=\inf _H rk(H)$ with the infimum taken over all subgroups of finite index.
 
 Two generating sets are Nielsen equivalent if one can pass from one to the other by Nielsen operations:  permuting the generators, replacing $g_i$ by $g_i\m$ or $g_ig_j$. For instance, any generating set of $\Z$ is Nielsen equivalent to $\{0,\dots,0,1\}$ by the Euclidean algorithm.
 
 The $\varphi$-orbit of $a\in A$ is $\{\varphi^n(a)\mid n\in\Z\}$. We denote by
$\Or(\varphi)$ the minimum number of $\varphi$-orbits needed to generate $A$. Clearly $\Or(\varphi)\le rk(A)$.
 We also denote by 
 $\vor(\varphi)$ the minimum number of $\varphi$-orbits needed to generate a finite index subgroup of $A$, so $\vor(\varphi)\le vrk(A)$.

  \begin{lem} Given $a,a_1,\dots,a_k\in A$, the intersection   $A'=  \langle a_1,\dots,a_k, at\rangle\cap A$   is generated by the $(i_a\circ \varphi)$-orbits of $a_1,\dots,a_k$.

 The $(i_a\circ \varphi)$-orbits of $a_1,\dots,a_k$ generate $A$ if and only if  $a_1,\dots,a_k, at$ generate $G$.
 \end{lem}
 
 \begin{proof} One has $(i_a\circ \varphi)^n(v)=(at)^nv(at)^{-n}$ for $v\in A$ and $n\in\Z$. This shows that the $(i_a\circ \varphi)$-orbit of $a_i$ is contained in $A'$. 
 Conversely, if  $v\in A'$, write it  in terms of $a_1,\dots,a_k, at$. The exponent sum of $t$ is 0, so $v$ is a product of elements of the form $(at)^na_i(at)^{-n}$.
 
 If $A'=A$, then $\langle a_1,\dots,a_k, at \rangle$ contains $A$ and $at$, so equals $G$. 
 \end{proof}
 
 \begin{cor} $rk(G)=1+\min_{a\in A}\Or(i_a\circ \varphi)$.
\end{cor}

\begin{proof}  $\le$  is clear. For the converse, use that  any finite generating set of $G$ is Nielsen equivalent to a set $\{a_1,\dots,a_k, at\}$ (Euclid's algorithm).
\end{proof}

 \begin{cor}   $vrk(G)=1+\min_{a\in A, n\ne0}\vor(i_a\circ \varphi^n)$.
\end{cor}

\begin{proof}  If $n\ne0$ and the $(i_a\circ \varphi^n)$-orbits of $a_1,\dots,a_k$ generate a finite index subgroup of  $A$, the subgroup of $G$ generated by $a_1,\dots,a_k, at^n$ has finite index because it maps  onto $n\Z$ and it meets $A$ in a subgroup of  finite index.

Any finite subset of $G$ generating a finite index subgroup is Nielsen equivalent to $\{a_1,\dots,a_k, at^n\}$ with $n\ne0$, and the $(i_a\circ\varphi^n)$-orbits of $a_1,\dots,a_k$ generate a finite index subgroup of $A$. \end{proof}

 \begin{cor} Suppose that  
    $A$ is abelian.
     \begin{enumerate}
  
\item 
 $rk(G)=1+ \Or(  \varphi)$ and $vrk(G)=1+ \vor(  \varphi )$. 
\item  $G$ has rank $\le2$ if and only if $A$  
is generated by a single $\varphi$-orbit. A pair $(a_1,at)$ generates $G$ if and only if the $\varphi$-orbit of $a_1$ generates $A$. 
\item $vrk(G)$ is computable.
 \end{enumerate}
\end{cor}

\begin{proof}  $i_a$ is the identity and $\vor( \varphi )\le \vor( \varphi^n)$, so 1 follows from previous results. 2 is clear. 

For 3, first suppose $A=\Z^d$. View $\varphi$ as an automorphism of the vector space $\Q^d$. Then $\vor(  \varphi )$ is    the minimum number of $\varphi$-orbits needed to generate $\Q^d$. This is computable (it is the number of blocks in the rational canonical form of $\varphi$). If $A$ has a torsion subgroup $T$, then $A/T\simeq \Z^d$ for some $d$. Let   $\bar \varphi$ be the automorphism  induced on  $\Z^d$. Then    $\vor(  \varphi )=\vor(\bar  \varphi )$ is computable.
\end{proof}

\section{Computability}

Suppose $A=\Z^d$ with $d\ge1$. We view $\varphi\in Aut(A)$ as  an automorphism of $\Z^d$ or as a matrix in $ GL(d,\Z)$. Its companion  matrix $M_\phi$ is the unique matrix of the form $$ \left(\begin{array}{ccccc}
0&&&&*\\
1&0&&&*\\
&\ddots&\ddots&&*\\
&&1&0&*\\
&&&1&*
 \end{array}\right)$$ having the same characteristic polynomial as $\varphi$ (the empty triangles are filled with 0's, and $*$ denotes an arbitrary integer).

  \begin{lem}  \label{comp} 
  Let $\varphi\in GL(d,\Z)$, with $d\ge1$. 
  \begin{enumerate}
  \item
  The following are equivalent:
  \begin{enumerate}
\item  $G=\Z^d\rtimes_\varphi \Z$  has rank 2;
\item  $\Z^d$ may be generated by a single $\varphi$-orbit;
\item There exists $a\in \Z^d$ such that $\{a,\varphi(a),\dots,\varphi^{p-1}(a)\}$ is a basis of $\Z^d$. 
\item  $\phi$ is conjugate  to its companion matrix $M_\varphi$ in $GL(d,\Z)$.
  \end{enumerate}
  \item Suppose that the $\varphi$-orbit of $a$ generates $\Z^d$. Then the $\varphi$-orbit of $b$ generates $\Z^d$ if and only if $b=h(a)$ where $h\in GL(d,\Z)$ commutes with $\varphi$. 
  \end{enumerate}
 \end{lem}

\begin{proof}   We already know that (a) is equivalent to (b).  If  $a$ is the first element of a basis of $\Z^d$ in which $\varphi$ is represented by the matrix $M_\phi$, then the basis is $\{a,\varphi(a),\dots,\varphi^{d-1}(a)\}$ and 
the $\varphi$-orbit of $a$ generates $\Z^d$, so $(d)\Rightarrow(c)\Rightarrow(b)$.

Conversely, suppose that the $\varphi$-orbit of $a$ generates $\Z^d$. By the Cayley-Hamilton theorem, $\Z^d$ is generated by $\{a,\varphi(a),\dots,\varphi^{d-1}(a)\}$. This set is a basis of $\Z^d$ in which $\varphi$ is represented by $M_\phi$. This proves 1.

To prove 2, suppose that $h$ commutes with $\varphi$, and define $b=h(a)$. The image of the basis $\{a,\varphi(a),\dots,\varphi^{d-1}(a)\}$ by $h$ is $\{b,\varphi(b),\dots,\varphi^{d-1}(b)\}$, so the orbit of $b$ generates. Conversely, if the orbit of $b$ generates, define $h$ as the automorphism taking $\{a,\varphi(a),\dots,\varphi^{d-1}(a)\}$ to $\{b,\varphi(b),\dots,\varphi^{d-1}(b)\}$. It commutes with $\varphi$ because $M_\varphi$ represents $\varphi$ in both bases.
\end{proof}

\begin{prop}\label{compu}
 If  
 $A$ is nilpotent, one can decide whether $G=A\rtimes_\varphi \Z$ has rank 2 or not.
\end{prop}

\begin{proof}  If $A=\Z^d$, one has to decide whether $\phi$ is conjugate to its companion matrix $M_\varphi$ in $GL(d,\Z)$. This is possible because the conjugacy problem is solvable in $GL(d,\Z)$ by \cite{grunewald}.

We now assume   that $A$ is abelian. It fits in an exact sequence $0\to T\to A\to \Z^d\to 0$ with $T$ finite. We denote by $a\mapsto \bar a$ the map $A\to \Z^d$, and by  $h\mapsto \bar h$ the natural epimorphism $Aut(A)\to Aut(\Z^d)$. They each have finite kernel. 

We have to decide whether $A$ may be generated by a single $\varphi$-orbit. We first check whether the matrix of $\bar\varphi$ is conjugate to its companion matrix.
If not, the answer to our question is no. If yes,  \cite{grunewald} yields  a conjugator and therefore an explicit $u\in\Z^d$ whose $\bar\varphi$-orbit  generates $\Z^d$. 

We claim that $A$ may be generated by a single $\varphi$-orbit if and only if there exist $a\in A$ mapping onto $u$, and $\psi\in Aut(A)$ of the form $h\varphi h\m$ with $h\in Aut(A)$ and $[\bar h,\bar\varphi]=1$, such that the $\psi$-orbit of $a$ generates $A$. 

The ``if'' direction is clear. Conversely, suppose that the $\varphi$-orbit of $b$ generates $A$. Then the $\bar\varphi$-orbit of $\bar b$ generates $\Z^d$, so by Lemma \ref{comp} there exists $\theta\in Aut(\Z^d)$ commuting with $\bar \varphi$ and mapping $\bar b$ to $u$.
Let $h$ be any lift of $\theta$ to $Aut(A)$. Defining    $a=h(b)$ and $\psi=h\varphi h\m$, it is easy to check that the $\psi$-orbit of $a$ generates $A$. This proves the claim.

We now explain how to decide whether $a$ and $\psi$ as above exist. Note that $a$ and $\psi$ must belong to   explicit finite sets: $a$   belongs to  the preimage $A_u$ of $u$, and 
$\psi$     belongs to     the preimage $X_\varphi$ of $\bar\varphi$ in $Aut(A)$. 

By Theorem C of \cite{grunewald}, the centralizer of $\bar\varphi$  in $Aut(\Z^d)$ is a finitely generated subgroup and one can compute a finite generating set. The same is true of 
  $D=\{h\in Aut(A)\mid [\bar h,\bar\varphi]=1\}$, so we can list the elements $\psi$ in the orbit $D\varphi$ of $\varphi$ for the action of $D$ on $X_\varphi$ by conjugation.  
  
By the claim proved above,    $A$ may be generated by a single $\varphi$-orbit if and only if there exist 
    $a \in A_u$ and $\psi\in D\varphi$  such that the $\psi$-orbit of $a$ generates $A$. To decide this, we enumerate the pairs $(a,\psi)$ with $a\in A_u$ and $\psi\in D\varphi$. For each pair,     
    we consider the increasing sequence of subgroups $A_N=\langle \psi^{-N}(a), \dots, \psi^{-1}(a), a,\psi(a), \dots\psi^{N}(a)\rangle$. It stabilizes and we check whether $A_N=A$ for $N$ large. 
    
    This completes the proof for $A$ abelian. If $A$ is nilpotent, let $B$ be its abelianization and let $\rho:B\to B$ be the automorphism induced by $\varphi$. If $G_\varphi=A\rtimes_\varphi\Z$ has rank 2, so does its quotient $G_\rho=B\rtimes_\rho\Z$. Conversely, if $G_\rho$ has rank 2, it is generated by $t$ and some $b\in B$ whose $\rho$-orbit generates $B$. Let $a$ be any lift of $b$ to $A$. The subgroup of $A$ generated by the $\varphi$-orbit of $a$ maps surjectively to $B$, so equals $A$ by a classical fact about nilpotent groups (see e.g.\  Theorem 2.2.3(d) of \cite{k}). Thus $G_\varphi$ has rank 2.
    \end{proof}
 
\begin{cor} If $A=\Z^2$ or $A=F_2$, one can compute the rank of $G$. 
\end{cor}

\begin{proof}  The rank is 2 or 3, so this is clear from the proposition  if $A=\Z^2$. 

Recall that the natural map $ Out(F_2)\to Out(\Z^2)=Aut(\Z^2)$ is an isomorphism (both groups are isomorphic to $GL(2,\Z)$).
Given $G=F_2\rtimes_\varphi \Z$,   let $\rho$ be the image of $\varphi$ in $Aut(\Z^2)$. 
Consider $G_\rho=\Z^2\rtimes_{\rho} \Z$. We prove that $G$ and $G_\rho$ have the same rank. 

Clearly $2\le rk(G_\rho)\le rk(G)\le3$. If $G_\rho$ has rank 2, Lemma \ref{comp} lets us assume that $\rho$ is of the form $ \left(\begin{array}{rr} 0&\pm1\\ 1&n  \end{array}\right)$. Since $G$ only depends on the class of $\varphi$ in $Out(F_2)$, it is isomorphic to $$\langle a,b,t\mid tat\m=b,tbt\m=a^{\pm1}b^n\rangle,$$ so has rank 2.
\end{proof}  

\section{Nielsen equivalence}

  \begin{prop} \label {gp} Suppose  
that  $A$ is abelian and $G=A\rtimes_\varphi\Z$ has rank 2. 
  \begin{enumerate}
  
\item  Any  generating pair of $G$  is Nielsen equivalent to a pair $(a,t)$ with $a\in A$.
  
\item  Two generating pairs  $(a,t)$ and $(b,t)$, with $a,b\in A$, are Nielsen equivalent if and only if $b$ belongs to the $\varphi$-orbit of $a$ or $a\m$.
\end{enumerate}
 \end{prop}

\begin{proof}   Given $x,y\in A$, and $n$, write
$$
(x,ty)\sim((ty)^nx(ty)^{-n},ty)=(\varphi^n(x),ty)
$$
and
$$
(x,ty)\sim (\varphi^n(x),ty)\sim(\varphi^n(x),ty\varphi^n(x))\sim( x,ty\varphi^n(x)).
$$

Every generating pair is equivalent to some $(a,ty)$, with the $\varphi$-orbit of $a$ generating $A$. But $(a,ty)\sim ( a,ty\varphi^n(a))$ so by an easy induction $(a,ty)\sim (a,t)$. This proves 1.

If $b=\varphi^n(a^\varepsilon)$ with $\varepsilon=\pm1$, then $(b,t)=(\varphi^n(a^\varepsilon),t)=(t^na^\varepsilon t^{-n},t)\sim(a,t)$. The converse follows from Theorem 2.1 of \cite{hw}.
We give a proof for completeness. If $(b,t)\sim(a,t)$, we can write $b=w(a,t)$ with $w$   a primitive word with exponent sum 0 in $t$. Such a word is conjugate to $a^{\pm1}$ in the free group $F(a,t)$, so $b$ is conjugate to $a^{\pm1}$ in $G$. Since $A$ is abelian, $b$ belongs to the $\varphi$-orbit of $a^{\pm1}$. 
\end{proof}

\begin{rem} More generally, if $A$ is abelian, any generating set of $G$  is Nielsen equivalent to a set of the form $\{a_1,\dots,a_k, t\}$.
\end {rem}

\begin{rem} \label{heis} The proposition does not extend to nilpotent groups. Let $A$ be the Heisenberg group $\langle a,b,c\mid [a,b]=c, [a,c]=[b,c]=1\rangle$. Let $\varphi$ map $a$ to $ab$ and $b$ to $b$. 
The generating pairs $(a,t)$ and $(ac\m,t)$ are Nielsen equivalent  (even conjugate) but $ac\m$ does not belong to the $\varphi$-orbit of $a^{\pm1}$.
Moreover, $(a, tc)$ is a generating pair which is not Nielsen equivalent to a pair $(x,t)$  with $x\in A$. 
Indeed, if it were, then $t$ would be conjugate to $tca^k$ for some $k\in\Z$ by \cite{hw}. Counting  exponent sum in $a$ yields $k=0$. But $t$ and $tc$ are not conjugate. 
\end {rem}

\begin{cor} Let $A=\Z^d$. If $G$  
has rank 2, the number of Nielsen classes of generating pairs is equal to the index of the  group generated by $\varphi$ and $-Id$ in the centralizer of $\varphi$ in $GL(d,\Z)$.
\end{cor}

\begin{proof}   By Proposition \ref{gp} we need only  consider generating pairs of the form $(a,t)$. Fix one. To any $b\in \Z^d$ such that  $(b,t)$ generates $G$ we associate the automorphism $\psi_b$ of $\Z^d$ taking the basis $\{a,\varphi(a),\dots,\varphi^{d-1}(a)\}$ to the basis $\{b,\varphi(b),\dots,\varphi^{d-1}(b)\}$. By Lemma \ref{comp}, the image of this map $b\mapsto \psi_b$ is the centralizer of $\phi$ in $GL(d,\Z)$. By Proposition \ref{gp}, $(b,t)\sim(a,t)$ if and only if $\psi_b$ is $\pm \varphi^n$ for some $n\in \Z$. 
\end{proof}

  \begin{example*}    The number of Nielsen classes of generating pairs is always finite if $d=2$. If $\phi=
  \left(\begin{array}{rrrr} 0&1&0&0\\ 1&1&0&0\\ 0 &0&0&1\\  0&0&1&0  \end{array}\right)$, this number is infinite.
  \end{example*}
  
\section{Powers}

Fix $\varphi\in GL(d,\Z)$.
Say that $v\in \Z^d$ is \emph{$\varphi$-cyclic} if 
  its $\varphi$-orbit generates $\Z^d$, or equivalently if $\{v,\phi(v),\dots,\phi^{d-1}(v)\}$ is a basis of $\Z^d$. 
  The existence of such a $v$   is equivalent to 
  $\phi$ being conjugate to its companion matrix, and also to $G$ having  rank 2.  If  $v$   is $\varphi^n$-cyclic  for some $n\ge2$, it  is   $\varphi$-cyclic since its $\varphi^n$-orbit is contained in its $\varphi$-orbit. 
  
  If $v$ is $\varphi$-cyclic, we denote by $\delta_n$ 
   the index of the subgroup of $\Z^d$ generated by the $\varphi^n$-orbit of $v$.
    It   does not depend on the choice of $v$ since $\varphi$ always has matrix $M_\phi$ in the basis $\{v,\varphi(v),\dots,\varphi^{d-1}(v)\}$. Also note that  $\delta_1=1$. The group $G_n=\Z^d\rtimes_{\varphi^n}\Z$ has rank 2 (equivalently, $\varphi^n$ is conjugate to its companion matrix) if and only if $\delta_n=1$.

\begin{thm} \label{po2}
If $\varphi\in GL(2,\Z)$ has infinite order, the rank of  $G_n=\Z^2\rtimes_{\varphi^n}\Z$   is 3 for all $n\ge3$. 
\end{thm}

\begin{proof}
  If $G_n$ has rank 2 for some $n$, there exists a $\varphi^n$-cyclic element $v$. Such a $v$ is also $\varphi$-cyclic.  In the basis $\{v,\varphi(v)\}$, the matrix of $\varphi$ has the form 
 $ M=\left(\begin{array}{rr} 0&\varepsilon \\ 1&\tau  \end{array}\right)$ with $\varepsilon=\pm1$. If finite, the index 
$\delta_n$  is the absolute value of the determinant $c_n$ of the matrix expressing the family $\{v,\varphi^n(v)\}$ in the basis $\{v,\varphi(v)\}$. We prove the theorem by showing   $ | c_n | >1$ for $n\ge3$.
 
 The number $c_n$ is determined by the equation $M^n=c_n M+ d_n I$. It follows from the Cayley-Hamilton theorem that the sequence $c_n$ satisfies the recurrence relation $c_{n+2}-\tau c_{n+1}-\varepsilon c_n=0$.  
 
  If $\varepsilon=-1$ one has $$c_n=\prod_{k=1}^{n-1}(\tau-2\cos\frac{k\pi}n) $$
 because $c_n$ is a monic polynomial  of degree $n-1$ in $\tau$ which vanishes for $\tau=2\cos\frac{k\pi}n$ 
  (one also has $c_n=U_{n-1}(\tau/2)$, with $U_{n-1}$ a Chebyshev polynomial of the second kind).

 If  $\varepsilon=1$ one has 
 $$c_n=\prod_{k=1}^{n-1}(\tau-2i\cos\frac{k\pi}n) .$$

 Since $\varphi$ is assumed to have infinite order, one has $\tau\ne0$ if $\varepsilon=1$, and $ | \tau | \ge2$ if $\varepsilon=-1$. One checks that $ | c_n | >1$ for $n\ge3$ (for $n\ge 2$ if $\varepsilon=-1$).
 
 \end{proof}

\begin{thm} \label{pow}
Suppose that $\varphi\in GL(d,\Z)$ has infinite order. 
\begin{enumerate}
\item There exists $n_0$ such that $G_n=\Z^d\rtimes_{\varphi^n}\Z$ has rank $\ge3$ for every $n\ge n_0$. Equivalently: $\varphi^n$ is not conjugate to its companion matrix for $n\ge n_0$.
\item 
More precisely, the minimum index of 2-generated subgroups of $G_n$ goes to infinity with $n$.
\end{enumerate}
\end{thm}

 Note that there are arbitrarily large values of $n$ for which the rank of $G_n$ is $d+1$ (whenever $\varphi^n$ is the identity modulo some prime number). As already mentioned, it is proved in \cite{AZ} that  $n_0$ may be chosen to depend only on $d$.
 
 The key step in the proof of Theorem \ref{pow} 
 is the following result.
  
  \begin{prop} \label{key}
  If $\varphi$ has infinite order and $v $ is $\varphi$-cyclic, then the index $\delta_n$ of the subgroup  of $\Z^d$ generated by the $\varphi^n$-orbit of $v$ goes to infinity with $n$.
  \end{prop} 
  
   \begin{proof}[Proof of the theorem from the proposition]
   As above, if $G_n$ has rank 2 for some $n$, there exists a $\varphi $-cyclic element $v$. 
  For $n$ large one has $\delta_n>1$, so $G_n$ has rank $>2$. Assertion 1 is proved.
   
 For Assertion 2, suppose that there are arbitrarily large values of $n$ such that $G_n$ contains a 2-generated subgroup $H_n$ of index $\le C$, for some fixed $C$. This subgroup has a generating pair of the form $(a_n,t_n)$ with $a_n \in\Z^d$, and the  intersection of $H_n$ with $\Z^d$ is generated by the $\varphi^{nm_n}$-orbit of $a_n$ for some $m_n\ge1$. It has index $\le C$ in $\Z^d$.

 The subgroup  of $\Z^d$ generated by the $\varphi$-orbit of $a_n$  has index $\le C$, so we can assume that it does not depend on $n$. Call it $J$. It is $\varphi$-invariant so we can apply the proposition to the action of $\varphi$ on $J$, with $v=a_n$. This gives the required contradiction.
 \end{proof}

    \begin{proof}  [Proof of Proposition \ref{key}]  When $d=2$, one  easily   checks that $  c_n  $, as computed above,  goes to infinity with $n$. The proof in the general case is more involved. 
    
    Define numbers $u_k(i)$, 
    for $k=0,\dots,d-1$ and $i\ge0$,   by $\varphi^i(v)=\sum_{k=0}^{d-1}u_k(i)\varphi^k(v)$. The sequences $u_0,\dots, u_{d-1}$ 
     form a basis for the 
    space $\cals$ of sequences satisfying the linear recurrence associated to the characteristic polynomial of $\varphi$ (the recurrence is $\sum_{j=0}^d a_ju_k(i+j)=0$ if the characteristic polynomial is $\sum _{j=0}^d  a_jX^j$). 
    
    The index $\delta_n$  is the absolute value of the determinant $c_n$ of the matrix $(u_k(ni))_{0\le i,k\le d-1}$ (it is infinite if the determinant is $0$). We have to prove that, given $c\ne0$, the set of $n$'s such that $c_n=c
   $ is finite. We assume it is not and we work towards a contradiction. 
   
   A sequence satisfies a linear recurrence if and only if it is a finite sum of polynomials times exponentials, so     $c_n$ also is a recurrent sequence. The Skolem-Mahler-Lech theorem \cite{rs} then implies that $c_n=c
   $ for all $n$ in an  arithmetic progression $\N_0\inc \N$.

    We shall now replace the basis $u_k$ of $\cals$ by another basis $w_k$ depending on the eigenvalues of $\varphi$. We  then assume  that    $D_n:=\det (w_k(ni))_{0\le i,k\le d-1}=c'\ne0$ for $n\in\N_0$. 

We order  the eigenvalues $\lambda_k$ of $\varphi$  so that $0<| \lambda_1 |\le | \lambda_2 |\le\dots\le | \lambda_{d} | $.
First suppose that the eigenvalues  are all distinct. We then choose $w_k(i)= (\lambda_{k+1})^i$. In this case $D_n$ is a Vandermonde determinant, for instance $$D_n=
\left |   \begin{matrix} 
1 & 1&1   \cr
(\lambda_1)^{n}& (\lambda_2)^{n} & (\lambda_3)^{n}  \cr  
(\lambda_1)^{2n} &(\lambda_2)^{2n} &(\lambda_3)^{2n}  
   \end{matrix}\right | $$ for $d=3$, 
   so $\displaystyle D_n=\prod_{1\le k< m\le d}\bigl((\lambda_m)^{n} -(\lambda_k)^{n} \bigr)$.
   
If all moduli $ | \lambda_k | $ are distinct, 
then $ | D_n | $ goes to infinity with $n$ because  its diagonal  term $$ (\lambda_2)^{n}      (\lambda_3)^{2n}  \dots(\lambda_{d})^{(d-1)n}=\biggl( \lambda_2      (\lambda_3)^{2}  \dots(\lambda_{d})^{(d-1)}\biggr)^n$$ has modulus bigger than all others.  

   If the $\lambda_k$'s are distinct but their moduli are not, 
    expand      $D_n$  as   a sum $\sum_j \varepsilon_j\mu_j{}^ n$ (with $\varepsilon_j=\pm1$). 
  Now  
 there may be several (possibly cancelling) terms for which $ | \mu_j | $ takes its maximal value $K= |  \lambda_2      (\lambda_3)^{2}  \dots(\lambda_{d})^{(d-1)} | $. Note that $K>1$ because otherwise all $\lambda_k$'s have modulus 1, hence are roots of unity by a classical result, and $\varphi$   has  finite order. 
   
 Since $D_n=c'$ for $n\in \N_0$ and $K>1$, one has    $\sum_{| \mu_j | =K} \varepsilon_j\mu_j{}^ n=0$ for $n\in\N_0$. Call this sum $D_{n,K}$. Recall that $\displaystyle D_n=\prod_{1\le k< m\le d}\bigl((\lambda_m)^{n} -(\lambda_k)^{n} \bigr)$. To expand  this product, one  chooses one of   $(\lambda_m)^{n}$ or 
$(\lambda_k)^{n} $ for each couple  $k,m$. The corresponding term contributes to $D_{n,K}$ if and only if one always chooses a term of maximal modulus. In other words, $\displaystyle D_{n,K}=\prod_{1\le k< m\le p} E_{k,m}$ with $E_{k,m}=(\lambda_m)^{n} -(\lambda_k)^{n}$ if $ | \lambda_m | = | \lambda_k | $ and
$E_{k,m}=(\lambda_m)^{n} $ if $ | \lambda_m | > | \lambda_k | $. Since the $\lambda_k$'s are non-zero, $D_{n,K}=0$ implies 
 $(\lambda_k)^n=(\lambda_m)^n$ for some $k,m$ with $k\ne m$, so that $D_n=0$, a contradiction.

    This completes the proof when the eigenvalues of $\varphi$ are distinct. In the remaining case, the basis $w_k$ must have a different form: if $\lambda$ is an eigenvalue of multiplicity $r$, we use the sequences $\lambda^i, i\lambda^i, \dots, i^{r-1}\lambda^i$. For instance,
    $$D_n=
\left |   \begin{matrix} 
1 & 0&0&1   \cr
(\lambda_1)^{n}& n(\lambda_1)^{n} & n^2(\lambda_1)^{n} & (\lambda_4)^{n}  \cr  
(\lambda_1)^{2n} &2n(\lambda_1)^{2n} &(2n)^2(\lambda_1)^{2n} &(\lambda_4)^{2n} \cr
(\lambda_1)^{3n} &3n(\lambda_1)^{3n} &(3n)^2(\lambda_1)^{3n} &(\lambda_4)^{3n}  
 
   \end{matrix}\right | $$
   when $d=4$ and $\lambda_1=\lambda_2= \lambda_3\ne \lambda_4$.
    
    Calling $\nu_1,\dots,\nu_q$ the distinct eigenvalues of $\varphi$, there exist   integers $a,b,c_k,d_{mk}$ (depending only on the multiplicities of the eigenvalues) such that $$D_n=an^b\prod _{k=1}^q(\nu_k)^{nc_k}\prod _{1\le k< m\le q}\bigl((\nu_m)^{n} -(\nu_k)^{n}\bigr)^{d_{mk}}$$ (see \cite{fh} or Theorem 21 in \cite{kr}). For instance, $D_n$ as displayed above equals $2n^3(\lambda_1)^{3n}((\lambda_4)^{n} -(\lambda_1)^{n})^3$.
    
    If $K>1$, we conclude as in the previous case. 
    If $K=1$,
  all eigenvalues are roots of unity and $D_n=n^bE_n$ where $E_n$ only takes finitely many values and $b>0$ (an eigenvalue $\nu_j$ of multiplicity $r\ge 2$ contributes   $1+\dots+(r-1)$ to $b$).  
Such a product cannot take a non-zero value infinitely often.
     \end{proof}

     \begin{cor} If $A$ is abelian, and $\varphi\in Aut(A)$ has infinite order, then 
     $G_n=A\rtimes_{\varphi^n}\Z$ has rank $\ge3$ for $n$ large. The minimum index of 2-generated subgroups of $G_n$ goes to infinity with $n$.
     \end{cor}
     
     This follows readily from Theorem \ref{pow}, writing $A/T\sim\Z^d$ with $T$ finite.
 The analogous result for nilpotent groups is false, as the following example shows.      
 Let $A$ be the Heisenberg group as in Remark \ref{heis}. If $\varphi$ maps $a$ to $bc$, $b$ to $ac^2$, and $c$ to $c\m$, then $\varphi^{2n+1}(a)=bc^{1-n}$, so  $G_{2n+1}$ has rank 2 since $a$ and $\varphi^{2n+1}(a)$ generate $A$. The automorphism induced by $\varphi$ on the abelianization  of $A$ has order 2.

\bigskip
\small
\begin{flushleft}
 
  Gilbert Levitt\\
  Laboratoire de Math\'ematiques Nicolas Oresme\\
  Universit\'e de Caen et CNRS (UMR 6139)\\
  BP 5186\\
  F-14032 Caen Cedex\\
  France\\
  \emph{e-mail: }\texttt{levitt@math.unicaen.fr}\\[5mm]

  Vassilis Metaftsis\\
University of the Aegean\\
Department of Mathematics\\
832 00 Karlovassi\\
Samos, Greece\\
  \emph{e-mail: }\texttt{vmet@aegean.gr}

\end{flushleft}


\begin{thebibliography}{9}

\bibitem{AZ} F.Amoroso, U. Zannier, in preparation.

\bibitem{sh}
G. Baumslag, C.F. Miller III,  H. Short,  Unsolvable problems about small cancellation and word hyperbolic groups,  Bull. London Math. Soc.  26  (1994),  97--101.  
\bibitem {rs}
G. Everest, A. van der Poorten, I. Shparlinski, T. Ward, Recurrence sequences, AMS	 Mathematical surveys and monographs 104, 2003.
\bibitem{fh}  R.P. Flowe, G.A. Harris, A note on generalized Vandermonde determinants, SIAM J. Matrix Anal. Appl 14 (1993), 1146-1151.
\bibitem{grunewald}  F. Grunewald, Solution of the conjugacy problem in certain arithmetic groups.  Word problems, II (Conf. on Decision Problems in Algebra, Oxford, 1976),  pp. 101--139, Stud. Logic Foundations Math., 95, North-Holland, Amsterdam-New York, 1980.
\bibitem {hw}M. Heusener, R. Weidmann, Generating pairs of 2-bridge knot groups,	arXiv:0902.0799.
\bibitem {kw} I.  Kapovich, R. Weidmann, Kleinian groups and the rank problem, Geometry and Topology 9 (2005),  375--402.
\bibitem{k}  E.I. Khukhro, Nilpotent groups and their automorphisms, de Gruyter expositions in mathematics 8, 1993.
\bibitem{kr}  C. Krattenthaler  Advanced determinant calculus. The Andrews Festschrift (Maratea, 1998).  S\'em. Lothar. Combin.  42  (1999), Art. B42q, 67 pp.
\bibitem{le} G. Levitt, in preparation.
\bibitem{souto} J. Souto,  The rank of the fundamental group of certain hyperbolic 3-manifolds fibering over the circle, in The Zieschang Gedenkschrift, Geometry and Topology Monographs, Vol. 14, 2008.

\end{thebibliography}
\end{document}